\documentclass{amsart}
\usepackage{amsmath, amsthm}
\usepackage{amsfonts}
\usepackage{amssymb}
\usepackage{amsthm}
\usepackage{amsmath}
\usepackage{latexsym}
\usepackage[all]{xy}
\usepackage{mathrsfs}
\usepackage{graphicx}

\newtheoremstyle{theorem}
  {10pt}		  
  {10pt}  
  {\sl}  
  {\parindent}     
  {\bf}  
  {. }    
  { }    
  {}     
\theoremstyle{theorem}
\newtheorem{theorem}{Theorem}
\newtheorem{corollary}[theorem]{Corollary}
\newtheorem{lemma}[theorem]{Lemma}
\newtheorem{prop}[theorem]{Proposition}

\newtheoremstyle{defi}
  {10pt}		  
  {10pt}  
  {\rm}  
  {\parindent}     
  {\bf}  
  {. }    
  { }    
  {}     
\theoremstyle{defi}

\title{A Lower Bound for the Size of a Sum of Dilates}
\author{\v Zeljka Ljuji\' c}
\address{Mathematics Ph.D. Program\\
The CUNY Graduate Center\\
365 Fifth Avenue\\
New York NY 10016}
\email{ zljujic@gc.cuny.edu}
\date{}

\begin{document}

\maketitle

\begin{abstract} Let $A$ be a subset of integers and let $2\cdot A+k\cdot A=\{2a_1+ka_2 : a_1,a_2\in A\}$. Y. O. Hamidoune and J. Ru\' e proved in \cite{HR} that if $k$ is an odd prime and $A$ a finite set of integers such that $|A|>8k^k$, then $|2\cdot A+k\cdot A|\ge (k+2)|A|-k^2-k+2$. In this paper, we extend this result for the case when $k$ is a power of an odd prime and the case when k is a product of two odd primes.
\end{abstract}

\section{Introduction}

Let $k$ be an integer and let $A$ be a finite set of integers. The $k$-dilation $k\cdot A$ of the set $A$ 
is the set  of all integers of the form $ka$, where $a\in A$. Let  $f(x_1,\ldots,x_n)=u_1x_1+\cdots+u_nx_n$ be a linear form with integer coefficients  $u_1,\ldots,u_n$. We define the set $f(A)=u_1\cdot A+\cdots+u_n\cdot A=\{u_1a_1+\cdots+u_ha_n: a_i\in A\}$. B. Bukh, in \cite{B} obtained the almost sharp lower bound for the size of the sets $f(A)$:
$|u_1\cdot A+\cdots u_n\cdot A|\ge (|u_1|+\cdots+ |u_n|)|A|-o(|A|)$, where $u_1,\ldots, u_n$ are integers such that $(u_1, \ldots, u_n)=1$. 

In the case of binary linear forms we write $f(x,y)=mx+ky$, where $m$ and $k$ are nonzero integers. We are interested in finding a sharp lower bound for $|f(A)|$. It is easy to see (\cite{MO}) that it is enough to consider only normalized binary linear forms satisfying $k\ge |m| \ge 1$ and $(m,k)=1$. Many authors (\cite{B},\cite{CHS},\cite{CSV},\cite{MN1}) studied the lower bounds of $|f(A)|$ for the case $m=1$. The sharp lower bound for $|A+k\cdot A|$ was known for the case $k=1$ (see \cite{MN}), and it was given for $k=2$ in \cite{MN1} and $k=3$ in \cite{CSV}. J. Cilleruelo, M. Silva, C. Vinuesa conjectured in  \cite{CSV} that if $k$ is a positive integer and $A$ a finite set of integers with sufficiently large cardinality, then $|A+k\cdot A|\ge (k+1)|A|-\lceil k(k+2)/4 \rceil$. This conjecture was proved for the case when $k$ is a prime number in \cite{CHS}, and very recently for the case when $k$ is a power of a prime and $k$ is a product of two primes in \cite{SHZ}. 

The case $m=2$ was studied in \cite{HR}. Y. O. Hamidoune and J. Ru\' e proved in \cite{HR} that if $k$ is an odd prime and $A$ a finite set of integers such that $|A|>8k^k$, then $|2\cdot A+k\cdot A|\ge (k+2)|A|-k^2-k+2$. In this paper, we extend this result for the case when $k$ is a prime power and a product of two primes. More precisely, we prove the following theorems.

\begin{theorem}\label{thm1} Let $A$ be a finite set of integers such that $|A|>8k^k$. If $k=p^\alpha$, where $p$ is an odd prime and $\alpha\in\mathbb{Z}_{\ge 1}$, then 
\[|2\cdot A+k\cdot A|\ge (k+2)|A|-k^2-k+2.
\]
\end{theorem}

\begin{theorem}\label{thm2} Let $A$ be a finite set of integers such that $|A|>8k^k$. If $k=pq$, where $p$ and $q$ are distinct odd primes, then 
\[|2\cdot A+k\cdot A|\ge (k+2)|A|-k^2-k+2.
\]
\end{theorem}

\section{Notation and Preliminaries}

Let $A$ be a finite set of integers and let $k$ be a positive integer. We define $\hat{A}$ to be the natural projection of the set $A$ on $\mathbb{Z}/k\mathbb{Z}$ and $c_k(A)=|\hat{A}|$. Then, if $c_k(A)=j$, we denote by $A_1$, $A_2, \ldots, A_j$ the distinct congruences classes of $A$ modulo $k$. We assume that $|A_1|\ge |A_2|\ge\ldots\ge|A_j|$. For every $1\le i\le j$, we write $A_i=kX_i+u_i$, where $0\le u_i<k$. Let $E=\{1\le i\le j\mid |X_i|<k\}$ and let $F=\{1\le i\le j\mid |X_i|=k\}$. We define the sets $\Delta_{ii}=(2A_i+k\cdot A)\setminus (2A_i+k\cdot A_i)$ for $1\le i\le j$.

\begin{lemma}[Chowla, \cite{MN}] \label{C} Let $n\ge 2$ and let $A$ and $B$ be nonempty subsets of $\mathbb{Z}/n\mathbb{Z}$. If $0\in B$ and $(b,n)=1$ for all $b\in B\setminus\{0\}$, then
\[|A+B|\ge \min\{n, |A|+|B|-1\}.
\]
\end{lemma}

The following proposition, as well as its corollaries and the following lemma are Proposition 3.2, Corollary 3.3, Corollary 3.4 and Lemma 4.1 from \cite{HR}.

\begin{prop}\label{l} Let $A$ and $B$ be finite set of integers and let $n$ and $m$ be coprime integer. Then
\[|n\cdot A+m\cdot B|\ge c_n(B)|A|+c_m(A)|B|-c_m(A)c_n(B).
\]
\end{prop}

\begin{corollary}\label{cor1} Let $2\le n<m$ be coprime integers. Let $A$ be a finite set of integers. Then $|n\cdot A+m\cdot B|\ge4|A|-4$.
\end{corollary}

\begin{corollary}\label{cor2} Let $k$ be an odd integer. Let $A$ be a finite set of integers such that $c_k(A)=k$. Then $|2\cdot A+k\cdot A|\ge(k+2)|A|-2k$.
\end{corollary}

\begin{lemma} \label{graph}Let $A$ be a finite set of integers and let $k$ be a positive integer. Then 
\[\sum_{i=1}^j \Delta_{ii}\ge j(j-1).
\]
\end{lemma}

In the proof of Theorem \ref{thm2}, we will use the following lemmas. They appear as Lemma 6 and Lemma 8 in \cite{SHZ}.

\begin{lemma}\label{l6} Let $k$ be a positive integer and let $A$ be a nonempty subset of $\mathbb{Z}/k\mathbb{Z}$. Let $\alpha$ be a nonzero element in $\mathbb{Z}/k\mathbb{Z}$. We have $A+\alpha=A$ if and only if
\[A=\bigcup_{\beta\in I} ((k,\alpha)\cdot\{0,1,\ldots,\frac{k}{(k,\alpha)}-1\}+\beta)
\]
for some nonempty set $I\subset\mathbb{Z}/(k,\alpha)\mathbb{Z}$ and $\frac{k}{(k,\alpha)}\mid|A|$.
\end{lemma} 

\begin{lemma}\label{l8} Let $k>2$ be an integer that is not a prime and let $A$ be a nonempty subset of $\mathbb{Z}/k\mathbb{Z}$. Let $(q,k)\ne 1$ and $0\in B\subset(\{0,\bar{q}\}\cup\{\bar{b}\mid(b,k)=1\})$. If $|A+\{0,\bar{q}\}|\ge |A|+1$, then
\[|A+B|\ge \min(k,|A|+|B|-1).
\]
\end{lemma}

\section{The case $k=p^{\alpha}$}

\begin{lemma}\label{2-full} Let $A$ be a finite set of integers such that $\gcd(A)=1$ and $0\in A$. Let $k=p^\alpha$, where $p$ is an odd prime number and $\alpha\in\mathbb{Z}_{\ge 1}$. If $|\Delta_{ii}|<|A_i|$, then $c_2(A_i)=2$.
\end{lemma}

\begin{proof} Let us assume that $c_2(A_i)=1$. Thus, $A_i$ contains only even or only odd integers. 

Let $A_i$ contains only even integers. There exists an odd $a\in A$, since $\gcd(A)=1$. Then
\[|\Delta_{ii}|= |(2\cdot A_i+k\cdot A)\setminus(2\cdot A_i+k\cdot A_i)|\ge|(2\cdot A_i+ka)|=|A_i|,
\]
a contradiction.

Similarly, if $A_i$ contains only odd integers
\[|\Delta_{ii}|= |(2\cdot A_i+k\cdot A)\setminus(2\cdot A_i+k\cdot A_i)|\ge |(2\cdot A_i+k0)|=|A_i|,
\]
a contradiction.
\end{proof}

\begin{lemma} \label{imp} Let $A$ be a finite set of integers such that $\gcd(A)=1$. Let $k=p^\alpha$, where $p$ is an odd prime number and $\alpha\in\mathbb{Z}_{\ge 1}$. Let $m=\min\{1\le i\le j\mid p\nmid u_i\}$ and $i\in E\setminus \{m\}$. 
\begin{enumerate} 
\item[(i)] If $p\mid u_i$, then $|\Delta_{ii}|\ge |A_m|$.
\item[(ii)] If $u_l=0$ and $p\nmid u_i$, then $|\Delta_{ii}|\ge |A_l|$.
\end{enumerate} 
\end{lemma}

\begin{proof} 
 
(i) We have 

\begin{equation} |\Delta_{ii}|=|(2\cdot  A_i+k\cdot A)\setminus(2\cdot A_i+k\cdot A_i)|\ge|(2\cdot X_i+A_m)\setminus(2\cdot X_i+A_i)|.\label{1}
\end{equation}
On the other hand $(u_m-u_i,k)=1$, so using Lemma \ref{C} and that $|X_i|<k$, we obtain
 \[|2\cdot \hat{X}_i+\{0,u_m-u_i\}|\ge |\hat{X}_i|+1,
\] 
thus
\begin{equation}|(2\cdot \hat{X}_i+u_m)\setminus(2\cdot \hat{X}_i+u_i)|\ge 1.\label{2}
\end{equation}
Combining \eqref{1} and \eqref{2}, we conclude
\[|\Delta_{ii}|\ge |A_m||(2\cdot \hat{X}_i+u_m)\setminus(2\cdot \hat{X}_i+u_i)|\ge |A_m|.
\]

(ii) Similarly as in (i),

\[ |\Delta_{ii}|=|(2\cdot A_i+k\cdot A)\setminus(2\cdot A_i+k\cdot A_i)|\ge|(2\cdot X_i+A_l)\setminus(2\cdot X_i+A_i)|.
\]
We have $(u_l-u_i,k)=1$, so
\[|(2\cdot \hat{X}_i+u_l)\setminus(2\cdot \hat{X}_i+u_i)|\ge 1
\]
and 
\[|\Delta_{ii}|\ge |A_l||(2\cdot \hat{X}_i+u_l)\setminus(2\cdot \hat{X}_i+u_i)|\ge |A_l|.
\]

\end{proof}

\begin{lemma}\label{da} Let $A$ be a finite set of integers. If $k=p^\alpha$, where $p$ is an odd prime and $\alpha\in\mathbb{Z}_{\ge 1}$, then 
\[|2\cdot A+k\cdot A|\ge (k+2)|A|-4k^{k-1}.
\]
\end{lemma}

\begin{proof} Let $T$ be the set of integers $t$ such that for every finite set $A\subset\mathbb{Z}$
\[|2\cdot A+k\cdot A|\ge (t+2)|A|-4k^{t-1}.
\]
We will use induction to prove $k\in T$. By Corollary \ref{cor1}, we obtain that $2\in T$. Let us assume that $2\le t\le k$ and $t-1\in T$. Let $A$ be a finite set of integers.

\emph{Case 1.} $\sum_{i\in E} |\Delta_{ii}|\ge \sum_{i\in E} |A_i|$

By Corollary \ref{cor2}, for every $i\in F$, we have $|2\cdot A_i+k\cdot A_i|\ge (k+2)|A_i|-2k$. On the other hand, if $i\in E$, using induction hypothesis we get $|2\cdot A_i+k\cdot A_i|\ge (t+1)|A_i|-4k^{t-2}$. Hence,

\begin{align*} 
|2\cdot A+k\cdot A| & = \sum_{i\in E}|2\cdot A_i+k\cdot A|+\sum_{i\in F}|2\cdot A_i+k\cdot A|\\ 
& \ge  \sum_{i\in E} (|2\cdot A_i+k\cdot A|+|\Delta_{ii}|)+\sum_{i\in F}|2\cdot A_i+k\cdot A_i|\\
& \ge  \sum_{i\in E}[(t+1)|A_i|-4k^{t-2}] +\sum_{i\in E}|\Delta_{ii}|+\sum_{i\in F}[(k+2)|A_i|-2k]\\
& \ge  \sum_{i\in E}[(t+1)|A_i|-4k^{t-2}] +\sum_{i\in E}|A_i|+\sum_{i\in F}[(k+2)|A_i|-2k]\\
& \ge (t+2)|A|-(4|E|k^{t-2}+2|F|k)\ge (t+2)|A|-4k^{t-1}
\end{align*} 

\emph{Case 2.} $\sum_{i\in E} |\Delta_{ii}|< \sum_{i\in E} |A_i|$. 

Without loss of generality we may assume that $\gcd(A)=1$ and $0\in A_1$. We define $n=\min\{i\in E\mid |\Delta_{ii}|<|A_i|\}$. By Lemma \ref{2-full}, we have $c_2(A_n)=2$. Let $m=\min\{1\le i\le j\mid p\nmid u_i\}$. By Lemma \ref{imp}, we have that $|\Delta_{ii}|\ge |A_m|$ for all $i\in E$. Note that $m\ne n$.

We have $m>n$. For if $m<n$, by Lemma \ref{imp}, we have that $|\Delta_{ii}|\ge |A_n|$ for all $i\in E$ such that $i\ge n$ and this leads to contradiction:
\begin{align*}
\sum_{i\in E} |\Delta_{ii}| & = \sum_{i\in E,i<n} |\Delta_{ii}|+\sum_{i\in E,i\ge n} |\Delta_{ii}|\\
& \ge \sum_{i\in E,i<n} |A_i|+\sum_{i\in E,i\ge n} |A_n|\ge \sum_{i\in E} |A_i|.
\end{align*}

Next, by the definition of $m$, we have $p\mid u_n,\ldots, p\mid u_{m-1}$, so $(u_n-u_m,k)=\cdots=(u_{m-1}-u_m,k)=1$. Using Lemma \ref{C}, we obtain
\[ |\hat{2\cdot X_m}+\{0, u_n-u_m, \ldots, u_{s}-u_m\}|\ge \min\{k,  |\hat{X}_m|+s-n+1\},\textrm{ for all }n\le s\le m-1.
\]
Let $s=n$. If $|\hat{X}_m|<k$, we have 
\[|(2\cdot \hat{X}_m+u_n)\setminus(2\cdot \hat{X}_m+u_m)|\ge 1\] 
and  
\[|(2\cdot X_m+A_n)\setminus(2\cdot X_m+A_m)|\ge |A_n|.\]
Now, let $n<s<m-1$ such that 
\[|(2\cdot X_m+(A_n\cup A_{n+1}\cup\ldots\cup A_s))\setminus(2\cdot X_m+A_m)|\ge |A_n|+|A_{n+1}|+\cdots+|A_s|\]
and let us assume that $|\hat{X}_m|+s-n+2\le k$. We have
\[ |(2\cdot \hat{X}_m+\{u_n, \ldots, u_{s}\})\setminus(2\cdot \hat{X}_m+u_m)|\ge s-n+1
\]
and 
\[ |(2\cdot \hat{X}_m+\{u_n, \ldots, u_{s},u_{s+1}\})\setminus(2\cdot \hat{X}_m+u_m)|\ge s-n+2,
\]
so
\[|(2\cdot X_m+(A_n\cup A_{n+1}\cup\ldots\cup A_{s+1})\setminus(2\cdot X_m+A_m)|\ge |A_n|+|A_{n+1}|+\cdots+|A_{s+1}|.\]

We distinguish two subcases.

\emph{Case 2a.} $|\hat{X}_m|+m-n\le k$.

We have 
\begin{align*} |\Delta_{mm}|&=|(2\cdot A_m+k\cdot A)\setminus (2\cdot A_m+k\cdot A_m)|\\
& = |(2\cdot X_m+A)\setminus (2\cdot X_m+A_m)|\\
& \ge |(2\cdot X_m+(A_n\cup A_{n+1}\cup\ldots\cup A_{m-1}))\setminus(2\cdot X_m+A_m)|\\
& \ge |A_n|+|A_{n+1}|+\cdots+|A_{m-1}|.
\end{align*} 
By Lemma \ref{imp}, we have $|\Delta_{ii}|\ge |A_m|$, for all $i\in E\setminus\{m\}$, so
\begin{align*} 
|2\cdot A+k\cdot A| & = \sum_{i\in E\setminus\{m\}}|2\cdot A_i+k\cdot A|+\sum_{i\in F\setminus\{m\}}|2\cdot A_i+k\cdot A|+|2\cdot A_m+k\cdot A|\\ 
& \ge  \sum_{i\in E\setminus\{m\}} (|2\cdot A_i+k\cdot A_i|+|\Delta_{ii}|)+\sum_{i\in F\setminus\{m\}}|2\cdot A_i+k\cdot A_i|\\
&\quad+|2\cdot A_m+k\cdot A_m|+|\Delta_{mm}|\\
& \ge  \sum_{i\in E\setminus\{m\}}[(t+1)|A_i|-4k^{t-2}] +\sum_{i\in E\setminus\{m\}}|\Delta_{ii}|+\sum_{i\in F\setminus\{m\}}[(k+2)|A_i|-2k]\\
&\quad +(t+1)|A_m|-4k^{t-2}+|A_n|+|A_{n+1}|+\cdots+|A_{m-1}|\\
& \ge  \sum_{i\in E\cup\{m\}}[(t+1)|A_i|-4k^{t-2}] +\sum_{i\in E\cup\{m\}}|A_i|+\sum_{i\in F\setminus\{m\}}[(k+2)|A_i|-2k]\\
& \ge (t+2)|A|-4(|E|+|F|)k^{t-2}\ge (t+2)|A|-4k^{t-1}.
\end{align*} 

\emph{Case 2b.} $|\hat{X}_m|+m-n> k$.

In this case
\[ |2\cdot \hat{X}_m+\{0,u_n-u_m, \ldots , u_{m-1}-u_m\}|=k
\]
and
\begin{equation}|(2\cdot X_m+(A_n\cup\ldots\cup A_m))\setminus(2\cdot X_m+A_n)|\ge (k-|\hat{X}_m|)|A_m|.\label{a}
\end{equation}
On the other hand, we have $c_2(X_n)=c_2(A_n)=2$, so by Proposition \ref{l}
\begin{equation}|2\cdot X_m+A_n|=|2\cdot X_m+k\cdot X_n| \ge 2|X_m|+|\hat{X}_m|(|X_n|-2)=2|A_m|+|\hat{X}_m|(|A_n|-2).\label{b}
\end{equation}
We have $|A_n|\ge |A_m|$. Thus, by \eqref{a} and \eqref{b},
\begin{align*} |2\cdot A_m+k\cdot A|&= |2\cdot X_m+A|\\
&\ge|2\cdot X_m+A_n|+|2\cdot X_m+(A_n\cup\ldots\cup A_m))\setminus(2\cdot X_m+A_n)|\\
& \ge 2|A_m|+|\hat{X}_m|(|A_n|-2)+(k-|\hat{X}_m|)|A_m|\\
& \ge (k+2)|A_m|+|\hat{X}_m|(|A_n|-|A_m|)-2k.
\end{align*} 
By the definition of $m$, we have $m\le p^{\alpha-1}+1$, so 
\[|\hat{X}_m|>k-m+n\ge k-m+1\ge p^{\alpha}-p^{\alpha-1}\ge p^{\alpha-1}\ge m-1.\]
Thus
\[ |2\cdot A_m+k\cdot A| \ge (k+2)|A_m|+m(|A_n|-|A_m|)-2k
\]
and
\begin{align*} 
|2\cdot A+k\cdot A| & = \sum_{i\in E\setminus\{m\}}|2\cdot A_i+k\cdot A|+\sum_{i\in F\setminus\{m\}}|2\cdot A_i+k\cdot A|+|2\cdot A_m+k\cdot A|\\ 
& \ge  \sum_{i\in E\setminus\{m\}} (|2\cdot A_i+k\cdot A_i|+|\Delta_{ii}|)+\sum_{i\in F\setminus\{m\}}|2\cdot A_i+k\cdot A_i|\\
&\quad+(k+2)|A_m|+m(|A_n|-|A_m|)-2k\\
& \ge  \sum_{i\in E\setminus\{m\}}[(t+1)|A_i|-4k^{t-2}] +\sum_{i\in E\setminus\{m\}}|\Delta_{ii}|+\sum_{i\in F\setminus\{m\}}[(k+2)|A_i|-2k]\\
&\quad+ (k+2)|A_m|+m(|A_n|-|A_m|)-2k\\
& \ge  \sum_{i\in E\setminus\{m\}}[(t+1)|A_i|-4k^{t-2}] +\sum_{i\in E\setminus\{m\}}|A_i|+\sum_{i\in F\setminus\{m\}}[(k+2)|A_i|-2k]\\
&\quad+ (k+2)|A_m|-2k\\
& \ge (k+2)|A|-4(|E|+|F|)k^{t-2}\ge (k+2)|A|-4k^{t-1}.
\end{align*} 

\end{proof}

\textbf{Proof of Theorem \ref{thm1}.}  If $j=k$, applying Corollary \ref{cor2}, we obtain $|2\cdot A+k\cdot A|\ge (k+2)|A|-2k\ge (k+2)|A|-k^2-k+2$. We assume $j< k$. Without loss of generality we also assume that $\gcd(A)=1$ and $0\in A_1$. We have $|A_1|\ge \frac{|A|}{j}> 8k^{k-1}$. Let $m=\min\{1\le i\le j\mid p\nmid u_i\}$. We distinguish two cases.

\emph{Case 1.} $E=\emptyset$.

By Corrolary \ref{cor2} and Lemma \ref{graph}, we have
\begin{align*} 
|2\cdot A+k\cdot A| & = \sum_{i=1}^j |2\cdot A_i+k\cdot A|\\
& = \sum_{i=1}^j (|2\cdot A_i+k\cdot A_i|+|\Delta_{ii}|)\\
& = \sum_{i=1}^j |2\cdot X_i+k\cdot X_i|+ \sum_{i=1}^j|\Delta_{ii}|\\
& \ge \sum_{i=1}^j [(k+2)|X_i|-2k]+j(j-1)\\
& = (k+2)|A|-j(2k-j+1)\\
& \ge (k+2)|A|-k^2-k+2.
\end{align*} 

\emph{Case 2.} $E\ne\emptyset$. 

We consider following subcases.

\emph{Case 2a.} $m\in E$

By Lemma \ref{imp}, we have $|\Delta_{mm}|\ge |A_1|$. Applying Lemma \ref{da}, we obtain
\begin{align*} 
|2\cdot A+k\cdot A| & \ge |2\cdot A_m+k\cdot A|+|2\cdot (A\setminus A_m)+k\cdot (A\setminus A_m)|\\
& = |2\cdot A_m+k\cdot A_m|+|\Delta_{mm}|+(k+2)|A\setminus A_m|-4k^{k-1}\\
& \ge (k+2)|A_m|-4k^{k-1}+|A_1|+(k+2)|A\setminus A_m|-4k^{k-1}\\
& >(k+2)|A|.\\
\end{align*} 

\emph{Case 2b.} $m\in F$. 

If $|\Delta_{11}|\ge |A_1|$, we have
\begin{align*} 
|2\cdot A+k\cdot A| & \ge |2\cdot A_1+k\cdot A|+|2\cdot (A\setminus A_1)+k\cdot (A\setminus A_1)|\\
& = |2\cdot A_1+k\cdot A_1|+|\Delta_{11}|+(k+2)|A\setminus A_m|-4k^{k-1}\\
& \ge (k+2)|A_1|-4k^{k-1}+|A_1|+(k+2)|A\setminus A_m|-4k^{k-1}\\
& >(k+2)|A|.\\
\end{align*}  

If $|\Delta_{11}|< |A_1|$, then by Lemma \ref{2-full}, we have $c_2(A_1)=2$. Since $E\ne\emptyset$, there exists $s\in E$. By Lemma \ref{imp}, we have $|\Delta_{ss}|\ge |A_m|$ if $p\mid u_s$ and $|\Delta_{ss}|\ge |A_1|$ if $p\nmid u_s$. Since $|A_1|\ge |A_m|$, we obtain $|\Delta_{ss}|\ge |A_m|$. We denote $A'=A\setminus (A_m\cup A_s)$. Applying Proposition \ref{l}, we obtain 
\begin{align*} 
|2\cdot A+k\cdot A| & \ge |2\cdot A_m+k\cdot A|+|2\cdot A_s+k\cdot A|+|2\cdot A'+k\cdot A'|\\
& \ge |2\cdot A_m+k\cdot A_1|+|2\cdot A_s+k\cdot A_s|+|\Delta_{ss}|+(k+2)|A'|-4k^{k-1}\\
& \ge 2|A_m|+k|A_1|-2k+(k+2)|A_s|-4k^{k-1}+|A_m|+(k+2)|A'|-4k^{k-1}\\
& =(k+2)|A|+|A_1|-8k^{k-1}-2k\\
& > (k+2)|A|-2k
\end{align*}

This ends the proof.

\section{The case $k=pq$}

\begin{lemma} \label{imp2} Let $A$ be a finite set of integers such that $\gcd(A)=1$. Let $k=pq$, where $p$ and $q$ are distinct odd prime numbers. Let $m=\min\{1\le i\le j\mid p\nmid u_i\}$ and let $i\in E$.
\begin{enumerate} 
\item[(i)] If $(u_2,k)=1$, then 
\begin{equation*} 
|\Delta_{ii}| \ge \left\{ 
\begin{array}{rl} 
|A_1| & \text{if } i=2\\ 
|A_2| & \text{if } i\ne 2
\end{array} \right. 
\end{equation*} 

\item[(ii)] If $(u_2,k)=p$, then 
\begin{equation*} 
|\Delta_{ii}| \ge \left\{ 
\begin{array}{ll} 
\min\{|A_2|,q|A_m|\} & \text{if } i=1\\ 
\min\{|A_1|,q|A_m|\} & \text{if } 1< i < m\\
|A_2| & \text{if } i=m\\ 
\min\{|A_1|,|A_2|,q|A_m|\}  & \text{if } i>m
\end{array} \right. 
\end{equation*} 

\end{enumerate} 
\end{lemma}

\begin{proof} 

(i) By Lemma \ref{l6}, if $i\in E$, we have $2\cdot \hat{X_i}=2\cdot \hat{X_i}+u_1\ne 2\cdot \hat{X_i}+u_2$. Otherwise, $k\mid |X_i|$, a contradiction. Thus, if $1\in E$, we have
\begin{align*}  |\Delta_{11}| & =|(2\cdot  A_1+k\cdot A)\setminus(2\cdot A_1+k\cdot A_1)|\ge|(2\cdot X_1+A_2)\setminus(2\cdot X_1+A_1)|\\
& \ge |A_2||(2\cdot \hat{X}_1+u_2)\setminus(2\cdot \hat{X}_1+u_1)|\ge |A_2|.
\end{align*} 
Similarly, if $2\in E$, we have $|\Delta_{11}| \ge |A_1|$. 

Now, let $i\in E$ and $i\ne 1,2$. Since $2\cdot \hat{X_i}+u_1\ne 2\cdot \hat{X_i}+u_2$, we have that $2\cdot \hat{X_i}+u_i\ne 2\cdot \hat{X_i}+u_1$, in which case $|\Delta_{ii}| \ge |A_1|$ or $2\cdot \hat{X_i}+u_i\ne 2\cdot \hat{X_i}+u_2$, in which case $|\Delta_{ii}| \ge |A_2|$. In both cases $|\Delta_{ii}| \ge |A_2|$.

(ii) Let $1\in E$. Then $2\cdot \hat{X_1}+u_1\ne 2\cdot \hat{X_1}+u_2$ or $2\cdot \hat{X_1}+u_1= 2\cdot \hat{X_1}+u_2$. If $2\cdot \hat{X_1}+u_1\ne 2\cdot \hat{X_1}+u_2$, we obtain, as in (i), that $|\Delta_{11}| \ge |A_2|$. If 
 $2\cdot \hat{X_1}+u_1= 2\cdot \hat{X_1}+u_2$, by Lemma \ref{l6}, we have that 
\[2\cdot \hat{X_1}=\bigcup_{\beta\in I} (p\cdot\{0,1,\ldots,q-1\}+\beta)
\]
for some nonempty set $I\subset\mathbb{Z}/p\mathbb{Z}$. Moreover, $p\nmid u_m$, thus $I+u_m\ne I$ and $|(2\cdot\hat{X_1}+u_m)\setminus (2\cdot\hat{X_1}+u_1)|\ge q$. We obtain
 \begin{align*}  |\Delta_{11}| & =|(2\cdot  A_1+k\cdot A)\setminus(2\cdot A_1+k\cdot A_1)|\ge|(2\cdot X_1+A_m)\setminus(2\cdot X_1+A_1)|\\
& \ge |A_m||(2\cdot \hat{X}_1+u_m)\setminus(2\cdot \hat{X}_1+u_1)|\ge q|A_m|.
\end{align*} 

Next, if $i<m$, we have that $p\mid u_i$ and $(k,u_i)=p$. As above, we have $2\cdot \hat{X_i}+u_i\ne 2\cdot \hat{X_i}+u_1$ or $2\cdot \hat{X_i}+u_i= 2\cdot \hat{X_i}+u_1$ and we obtain $|\Delta_{ii}|\ge |A_1|$ or $|\Delta_{ii}|\ge q|A_m|$.

If $m\in E$, we have $p\nmid u_m$. Thus, $q\nmid u_m$, in which case $(k,u_m)=1$, or $q\mid u_m$, in which case $(k,u_m-u_2)=1$. Thus, $2\cdot \hat{X_m}+u_m\ne 2\cdot \hat{X_m}+u_1$ or $2\cdot \hat{X_m}+u_m\ne 2\cdot \hat{X_m}+u_2$. We have 
 \begin{align*}  |\Delta_{mm}| =|(2\cdot X_m+A)\setminus(2\cdot X_1+A_m)|\ge |A_2|.
\end{align*} 

Finally, if $i>m$, we have $(k,u_i)=1$ or $(k,u_i)=p$ or $(k,u_i)=q$. If $(k,u_i)=1$, we have  $|\Delta_{ii}|\ge |A_1|$. If $(k,u_i)=p$, we obtain $|\Delta_{ii}|\ge |A_1|$ or $|\Delta_{ii}|\ge q|A_m|$. If $(k,u_i)=q$, we have $(k,u_m-u_2)=1$ and $|\Delta_{ii}|\ge |A_2|$.

\end{proof}

\begin{lemma} Let $A$ be a finite set of integers. If $k=pq$, where $p$ and $q$ are distinct odd primes, then 
\[|2\cdot A+k\cdot A|\ge (k+2)|A|-4k^{k-1}.
\]
\end{lemma}

\begin{proof} Let $T$ be the set of integers $t$ such that for every finite set $A\subset\mathbb{Z}$
\[|2\cdot A+k\cdot A|\ge (t+2)|A|-4k^{t-1}.
\]
As in the proof of Lemma \ref{da}, we will use induction to prove $k\in T$. By Corollary \ref{cor1}, we have that $2\in T$. Let us assume that $2\le t\le k$ and $t-1\in T$. Let $A$ be a finite set of integers. Without loss of generality we may assume that $\gcd(A)=1$ and that $0\in A_1$. We define $m=\min\{1\le i\le j\mid p\nmid u_i\}$.

If $\sum_{i\in E} |\Delta_{ii}|\ge \sum_{i\in E} |A_i|$ the same proof holds as in Lemma \ref{da}. Let us assume that $\sum_{i\in E} |\Delta_{ii}|< \sum_{i\in E} |A_i|$. We define $n=\min\{i\in E\mid |\Delta_{ii}|<|A_i|\}$.

\emph{Case 1.} $(u_2,k)=1$. We have $2\in F$. Otherwise, $2\in E$ and by Lemma \ref{imp2}, we have $\sum_{i\in E} |\Delta_{ii}|\ge \sum_{i\in E} |A_i|$, a contradiction. Moreover, since $|\Delta_{ii}| \ge |A_2|$ for all $i\in E$, we obtain that $1\in E$ and $|\Delta_{11}|<|A_1|$. By Lemma \ref{2-full}, we have $c_2(A_1)=2$. We obtain

\begin{align*} 
|2\cdot A+k\cdot A| & = \sum_{i\in E}|2\cdot A_i+k\cdot A|+\sum_{i\in F\setminus\{2\}}|2\cdot A_i+k\cdot A|+|2\cdot A_2+k\cdot A|\\ 
& \ge  \sum_{i\in E} (|2\cdot A_i+k\cdot A_i|+|\Delta_{ii}|)+\sum_{i\in F\setminus\{2\}}|2\cdot A_i+k\cdot A_i|\\
&\quad+|2\cdot A_2+k\cdot A_1|\\
& \ge  \sum_{i\in E}[(t+1)|A_i|-4k^{t-2}] +\sum_{i\in E}|\Delta_{ii}|+\sum_{i\in F\setminus\{2\}}[(k+2)|A_i|-2k]\\
&\quad +2|A_2|+k|A_1|-2k\\
& \ge  \sum_{i\in E}[(t+1)|A_i|-4k^{t-2}] +\sum_{i\in E}|A_2|+\sum_{i\in F\setminus\{2\}}[(k+2)|A_i|-2k]\\
&\quad +(k+1)|A_2|+|A_1|-2k\\
& \ge (t+2)|A|-4(|E|+|F|)k^{t-2}\ge (t+2)|A|-4k^{t-1}.
\end{align*} 

\emph{Case 2.} $(u_2,k)=p$. Thus $m\ge 3$. By Lemma \ref{2-full}, we have $c_2(A_n)=2$. By Lemma \ref{imp2}, we have $|\Delta_{ii}|\ge |A_m|$ for all $i\in E$. In particular $m\ne n$. Similarly as in Lemma \ref{da}, we obtain $m>n$. We have
\begin{align*} 
|2\cdot A+k\cdot A| & = \sum_{i\in E\setminus\{m\}}|2\cdot A_i+k\cdot A|+\sum_{i\in F\setminus\{m\}}|2\cdot A_i+k\cdot A|+|2\cdot A_m+k\cdot A|\\ 
& \ge  \sum_{i\in E\setminus\{m\}} (|2\cdot A_i+k\cdot A_i|+|\Delta_{ii}|)+\sum_{i\in F\setminus\{m\}}|2\cdot A_i+k\cdot A_i|\\
&\quad+|2\cdot X_m+A|\\
& \ge  \sum_{i\in E\setminus\{m\}}(t+1)|A_i| +\sum_{i\in E\setminus\{m\}}|\Delta_{ii}|+\sum_{i\in F\setminus\{m\}}(t+2)|A_i|+|2\cdot X_m+A|\\
&\quad-(\sum_{i\in E\setminus\{m\}}4k^{t-2}+\sum_{i\in F\setminus\{m\}}2k).
\end{align*}  

If $m\in F$, using Proposition \ref{l}, we obtain 
\begin{align*} 
|2\cdot X_m+A|&\ge |2\cdot X_m+A_n|=|2\cdot X_m+k\cdot X_n|\ge 2|X_m|+k|X_n|-2k\\
&=(k+2)|A_m|+k(|A_n|-|A_m|)-2k.
\end{align*} 
Thus
\begin{align*} 
|2\cdot A+k\cdot A| & \ge  \sum_{i\in E}(t+1)|A_i| +\sum_{i\in E}|\Delta_{ii}|+\sum_{i\in F\setminus\{m\}}(t+2)|A_i|\\
&\quad+(t+2)|A_m|+k(|A_n|-|A_m|)-(|E|4k^{t-2}+|F|2k)\\
&\ge (t+2)|A|-4k^{t-1}.
\end{align*}  

Next, let us assume $m\in E$. We have
\begin{align} \label{last}
|2\cdot A+k\cdot A| &\ge  \sum_{i\in E\setminus\{m\}}(t+1)|A_i| +\sum_{i\in E\setminus\{m\}}|\Delta_{ii}|+\sum_{i\in F}(t+2)|A_i|+|2\cdot X_m+A|\\
&\quad-(\sum_{i\in E\setminus\{m\}}4k^{t-2}+\sum_{i\in F}2k).\notag
\end{align}  
If $|A_1|\le q|A_m|$, using Lemma \ref{imp2}, we obtain that $1\in E$ and 
\[\sum_{i\in E} |\Delta_{ii}|\ge |A_2|+ \sum_{i\in E, i\ge 2} |A_i|.
\]
In particular, $|\Delta_{11}|< |A_1|$. Moreover, $2\not\in E$, otherwise, by Lemma \ref{imp2}, $|\Delta_{ii}|\ge |A_1|$ and $\sum_{i\in E} |\Delta_{ii}|\ge \sum_{i\in E} |A_i|$. Using the same argument as in the \emph{Case 1}, we obtain
\[|2\cdot A+k\cdot A| \ge (t+2)|A|-4k^{t-1}.
\]

We assume $|A_1|> q|A_m|$. Then $|A_n|> q|A_m|$. Otherwise, $n\ge 2$ and by Lemma \ref{imp2}, we have $|\Delta_{ii}|\ge |A_n|$, for all $i\in E$ and $\sum_{i\in E} |\Delta_{ii}|\ge \sum_{i\in E} |A_i|$, a contradiction. By Lemma \ref{imp2}, $|\Delta_{11}|\ge \min\{|A_2|,q|A_m|\}$, $|\Delta_{ii}|\ge q|A_m|$ for all $i\in E$ such that $1<i<m$ and $|\Delta_{ii}|\ge |A_m|$ for all $i\in E$ such that $i\ge m$. We need to consider separately the cases $|\hat{X}_m|<p$ and $|\hat{X}_m|\ge p$. Moreover, the case $|\hat{X}_m|\ge p$, we will subdivided in three subcases: $p\le |\hat{X}_m|< q$, $|\hat{X}_m|\ge p>q$ and $|\hat{X}_m|\ge q> p$. We will use that $m\le q+1$.

\emph{Case 2a.} $|\hat{X}_m|\ge p>q$. By Corollary \ref{cor2}, we have 
\begin{align*} |2\cdot X_m+A| & \ge |2\cdot X_m+A_n|\ge 2|X_m|+|\hat{X}_m||X_n|-2k\\
& \ge (k+2)|A_m|+p(|A_n|-q|A_m|)-2k\\
& \ge (t+2)|A_m|+(m-1)(|A_n|-q|A_m|)-2k.
\end{align*} 
If $n>1$, by (\ref{last}), we have
\begin{align*} 
|2\cdot A+k\cdot A| & \ge  \sum_{i\in E\setminus\{m\}}(t+1)|A_i|+\sum_{i\in E\setminus\{m\}}|\Delta_{ii}|+\sum_{i\in F}(t+2)|A_i| \\
&\quad+(t+2)|A_m|+(m-1)(|A_n|-q|A_m|)\\
&\quad-((|E|-1)4k^{t-2}+(|F|+1)2k)\\
&\ge (t+2)|A|-4k^{t-1}.
\end{align*}  
If $n=1$, then $c_2(A_1)=2$. We need to consider following subcases. 

If $2\in E$, by Lemma \ref{imp2}, we have that $|\Delta_{11}|\ge \min\{|A_2|,q|A_m|\}$ and $|\Delta_{22}|\ge \min\{|A_1|,q|A_m|\}$, so the above proof holds.

If $2\in F$, using Proposition \ref{l}, we obtain 
\[
|2\cdot X_2+A|\ge |2\cdot X_2+A_1|=|2\cdot X_2+k\cdot X_1|\ge 2|A_2|+k|A_1|-2k
\]
so
\begin{align*} 
|2\cdot A+k\cdot A| & = \sum_{i\in E\setminus\{m\}}|2\cdot A_i+k\cdot A|+\sum_{i\in F\setminus\{2\}}|2\cdot A_i+k\cdot A|\\
&\quad +|2\cdot A_2+k\cdot A|+|2\cdot A_m+k\cdot A|\\ 
& \ge  \sum_{i\in E\setminus\{m\}} (|2\cdot A_i+k\cdot A_i|+|\Delta_{ii}|)+\sum_{i\in F\setminus\{2\}}|2\cdot A_i+k\cdot A_i|\\
&\quad+|2\cdot X_2+A|+|2\cdot X_m+A|\\
& \ge  \sum_{i\in E\setminus\{m\}}(t+1)|A_i| +\sum_{i\in E\setminus\{m\}}|\Delta_{ii}|+\sum_{i\in F\setminus\{2\}}(t+2)|A_i|\\
&\quad+2|A_2|+k|A_1|+(t+2)|A_m|+(m-1)(|A_1|-q|A_m|)\\
&\quad-((|E|-1)4k^{t-2}+(|F|+1)2k)\\
&\ge (t+2)|A|-4k^{t-1}.
\end{align*}

\emph{Case 2b.} $|\hat{X}_m|\ge q>p$. Similarly as in previous case, we obtain 
\begin{align*} |2\cdot X_m+A| & \ge |2\cdot X_m+A_n|\ge 2|X_m|+|\hat{X}_m||X_n|-2k\\
& \ge (k+2)|A_m|+q(|A_n|-p|A_m|)-2k\\
& \ge (t+2)|A_m|+(m-1)(|A_n|-q|A_m|)-2k
\end{align*} 
and
\[
|2\cdot A+k\cdot A| \ge (t+2)|A|-4k^{t-1}.
\]

\emph{Case 2c.} $|\hat{X}_m|<p$. We have $|\hat{X}_m|+m-1<p+q\le pq=k$. Let $L=\{1\le i\le m-1\mid (u_i-u_m,k)\ne 1\}$. If $L=\emptyset$, then $(u_1-u_m,k)=\cdots=(u_{m-1}-u_m,k)=1$. Using Lemma \ref{C}, we obtain
\[ |2\cdot\hat{X}_m+\{0, u_1-u_m, \ldots, u_{s}-u_m\}|\ge |\hat{X}_m|+s-1,\textrm{ for all }1\le s\le m-1.
\]
Let $s=1$. We have 
\[|(2\cdot \hat{X}_m+u_1)\setminus(2\cdot\hat{X}_m+u_m)|\ge 1\] 
and  
\[|(2\cdot X_m+A_1)\setminus(2\cdot X_m+A_m)|\ge |A_1|.\]
Now, let $2\le s\le m-1$ such that 
\[|(2\cdot X_m+(A_1\cup A_{2}\cup\ldots\cup A_{s-1}))\setminus(2\cdot X_m+A_m)|\ge |A_1|+|A_{2}|+\cdots+|A_{s-1}|.\]
We have
\[ |(2\cdot\hat{X}_m+\{u_1, \ldots, u_{s-1}\})\setminus(2\cdot \hat{X}_m+u_m)|\ge s-1
\]
and 
\[ |(2\cdot\hat{X}_m+\{u_1, \ldots, u_{s-1},u_{t}\})\setminus(2\cdot \hat{X}_m+u_m)|\ge s,
\]
so
\[|(2\cdot X_m+(A_1\cup A_{2}\cup\ldots\cup A_{s})\setminus(2\cdot X_m+A_m)|\ge |A_1|+|A_{2}|+\cdots+|A_{s}|.\]
We have
\[|(2\cdot X_m+(A_1\cup A_{2}\cup\ldots\cup A_{m-1})\setminus(2\cdot X_m+A_m)|\ge |A_1|+|A_{2}|+\cdots+|A_{m-1}|.\]
Hence,
\begin{align*} |\Delta_{mm}|&=|(2\cdot A_m+k\cdot A)\setminus (2\cdot A_m+k\cdot A_m)|\\
& = |(2\cdot X_m+A)\setminus (2\cdot X_m+X_m)|\\
& \ge |(2\cdot X_m+(A_1\cup A_{2}\cup\ldots\cup A_{m}))\setminus(2\cdot X_m+A_m)|\\
& \ge |A_1|+|A_{2}|+\cdots+|A_{m-1}|
\end{align*} 
and 
\begin{align*} |2\cdot A_m+k\cdot A| & \ge |2\cdot A_m+k\cdot A_m|+|\Delta_{mm}|\\
& \ge (t+1)|A_m|+|A_1|+|A_{2}|+\cdots+|A_{m-1}|-4k^{t-2}.
\end{align*} 
Using (\ref{last}), we obtain
\begin{align*} 
|2\cdot A+k\cdot A| & \ge  \sum_{i\in E}(t+1)|A_i|+\sum_{i\in E\setminus\{m\}}|\Delta_{ii}|+\sum_{i\in F}(t+2)|A_i| \\
&\quad+|A_1|+|A_{2}|+\cdots+|A_{m-1}|-(|E|4k^{t-2}+|F|2k)\\
&\ge (t+2)|A|-4k^{t-1}.
\end{align*}

Now, let us assume that $L\ne \emptyset$. Thus there exists $1\le l\le m-1$ such that $(u_l-u_m,k)\ne 1$. Since $p\mid u_l$ and $p\nmid u_m$, we obtain that $(u_l-u_m,k)=q$. Thus $|L|=1$. Since $|X_m|< p$, by Lemma \ref{l6}, we have $|\hat{X}_m+(u_l-u_m)\setminus\hat{X}_m|\ge 1$. Then, using Lemma \ref{C} and Lemma \ref{l8}, we obtain
\[ |2\cdot\hat{X}_m+\{0,u_1-u_m, \ldots, u_{s}-u_m\}|\ge |\hat{X}_m|+s-1,\textrm{ for all }1\le s\le m-1.
\]
Similarly as in the case $L= \emptyset$, we have
\[|(2\cdot X_m+(A_1\cup A_{2}\cup\ldots\cup A_{m})\setminus(2X_m+A_m)|\ge |A_1|+|A_{2}|+\cdots+|A_{m-1}|\]
and 
\[
|2\cdot A+k\cdot A| \ge (t+2)|A|-4k^{t-1}.
\]

\emph{Case 2d.} $p\le|\hat{X}_m|<q$. Let $(X_m)_q=\{x(\textrm{mod }q)\mid q\in X_m\}$. Then $|(X_m)_q|\le |\hat{X}_m|<q$. Moreover, $(u_i,q)=1$, for $2\le i\le m-1$ and $|\{u_i(\textrm{mod }q)\mid 2\le i\le m-1\}|=m-2$, so by Lemma \ref{C}
\[ |2\cdot(X_m)_q+\{u_1(\textrm{mod }q),u_2(\textrm{mod }q),\ldots,u_t(\textrm{mod }q)\}|\ge\min\{q, |(X_m)_q|+s-1\}\]
for all $1\le s\le m-1$. Similarly, as in the previous case, we have
\[|(X_m+(A_1\cup A_2\cup\cdots\cup A_m)\setminus (X_m+A_1)|\ge|A_2|+\cdots+|A_r|,\]
where $r=\min\{m-1,q+1-(X_m)_q\}$. We obtain
\begin{align*} |2\cdot X_m+A| & \ge |2\cdot X_m+A_1|+|(X_m+(A_1\cup A_2\cup\cdots\cup A_m)\setminus (X_m+A_1)|\\
& \ge c_2(A_1)|X_m|+|\hat{X}_m||A_1|+|A_2|+\cdots+|A_r|-2k.
\end{align*} 
We have two subcases: $|\Delta_{11}|\ge |A_1|$ or $|\Delta_{11}|<|A_1|$ and $c_2(A_1)=2$. In both subcases, using Lemma \ref{imp2}, we obtain

\begin{align*} 
|2\cdot A+k\cdot A| & = \sum_{i\in E\setminus\{m\}}|2\cdot A_i+k\cdot A|+|2\cdot A_m+k\cdot A|+\sum_{i\in F}|2\cdot A_i+k\cdot A|\\ 
& \ge  \sum_{i\in E\setminus\{m\}} |2\cdot A_i+k\cdot A_i|+\sum_{i\in E\setminus\{m\}}|\Delta_{ii}|+|2\cdot X_m+A|\\
&\quad+\sum_{i\in F,i\le m-1}|2\cdot A_i+k\cdot A_n|+\sum_{i\in F,i>m}|2\cdot A_i+k\cdot A_i|\\
& \ge  \sum_{i\in E\setminus\{m\}}((t+1)|A_i|-4k^{t-1})\\
&\quad+|\Delta_{11}| +\sum_{i\in E,2\le i\le m-1}q|A_m|+\sum_{i\in E,m+1\le i\le j}|A_m|\\
&\quad+\sum_{i\in F,i\le m-1} (2|A_i|+k|A_n|-2k)+\sum_{i\in F,i>m}((k+2)|A_i|-2k)\\
&\quad+c_2(A_1)|X_m|+|\hat{X}_m||A_1|+|A_2|+\cdots+|A_r|-2k\\
& \ge  \sum_{i\in E\setminus\{m\}}(t+1)|A_i| +\sum_{i\in E,2\le i\le m-1}q|A_m|+\sum_{i\in E,m+1\le i\le j}|A_m|\\
&\quad+\sum_{i\in F,i\le m-1} (2|A_i|+k|A_n|)+\sum_{i\in F,i>m}(k+2)|A_i|\\
&\quad+2|X_m|+|\hat{X}_m||A_1|+|A_2|+\cdots+|A_r|-2k\\
&\quad -((|E|-1)4k^{t-2}+(|F|+1)2k)\\
& \ge \sum_{i\in E\setminus\{m\}}(t+2)|A_i|+\sum_{i\in F}(t+2)|A_i|+2|A_m|+(|\hat{X}_m|-m+r)|A_1|\\
&\quad +(m-2)q|A_m|-((|E|-1)4k^{t-2}+(|F|+1)2k)\\
& \ge \sum_{i\in (E\cup F)\setminus\{m\}}(t+2)|A_i|+2|A_m|+(|\hat{X}_m|+r-2)q|A_m|\\
&\quad -((|E|-1)4k^{t-2}+(|F|+1)2k).\\
\end{align*}  
By the definition of $r$, we have
\begin{align*} |\hat{X}_m|+r-2 &\ge \min\{|\hat{X}_m|+m-3,|\hat{X}_m|+q-1-(X_m)_q\}\\
&\ge\min\{|\hat{X}_m|+m-3,q-1\}\ge p.
\end{align*} 
Thus
\[|2\cdot A+k\cdot A| \ge (t+2)|A|-4k^{t-1}.
\]

\end{proof}

\textbf{Proof of Theorem \ref{thm2}.}  If $j=k$, applying Corollary \ref{cor2}, we obtain $|2\cdot A+k\cdot A|\ge (k+2)|A|-2k\ge (k+2)|A|-k^2-k+2$. We assume $j< k$. Without loss of generality we also assume that $\gcd(A)=1$ and $0\in A_1$. We have $|A_1|\ge \frac{|A|}{j}> 8k^{k-1}$. Let $m=\min\{1\le i\le j\mid p\nmid u_i\}$. 

The proof in the case $E=\emptyset$ is the same as the proof of this case in Theorem \ref{thm1}. We assume $E\ne\emptyset$. 

If $|\Delta_{11}|\ge |A_1|$, we have
\begin{align*} 
|2\cdot A+k\cdot A| & \ge |2\cdot A_1+k\cdot A|+|2\cdot (A\setminus A_1)+k\cdot (A\setminus A_1)|\\
& = |2\cdot A_1+k\cdot A_1|+|\Delta_{11}|+(k+2)|A\setminus A_m|-4k^{k-1}\\
& \ge (k+2)|A_1|-4k^{k-1}+|A_1|+(k+2)|A\setminus A_m|-4k^{k-1}\\
& >(k+2)|A|.\\
\end{align*}  

We assume $|\Delta_{11}|< |A_1|$. Then by Lemma \ref{2-full}, we have $c_2(A_1)=2$. We consider following cases.

\emph{Case 1.} $(u_2,k)=1$

Let $2\in F$. Since $E\ne\emptyset$, there exists $s\in E$. By Lemma \ref{imp2}, we have $|\Delta_{ss}|\ge |A_2|$. We denote $A'=A\setminus (A_2\cup A_s)$. We obtain
\begin{align*} 
|2\cdot A+k\cdot A| & \ge |2\cdot A_2+k\cdot A|+|2\cdot A_s+k\cdot A|+|2\cdot A'+k\cdot A'|\\
& \ge |2\cdot A_2+k\cdot A_1|+|2\cdot A_s+k\cdot A_s|+|\Delta_{ss}|+(k+2)|A'|-4k^{k-1}\\
& \ge 2|A_2|+k|A_1|-2k+(k+2)|A_s|-4k^{k-1}+|A_2|+(k+2)|A'|-4k^{k-1}\\
& \ge (k+2)|A|+|A_1|-8k^{k-1}-2k\\
& > (k+2)|A|-2k
\end{align*} 
If $2\in E$, then by Lemma \ref{imp2}, we have $|\Delta_{22}|\ge |A_1|$. Thus
\begin{align*} 
|2\cdot A+k\cdot A| & \ge |2\cdot A_2+k\cdot A|+|2\cdot (A\setminus A_2)+k\cdot (A\setminus A_2)|\\
& = |2\cdot A_2+k\cdot A_2|+|\Delta_{22}|+(k+2)|A\setminus A_2|-4k^{k-1}\\
& \ge (k+2)|A_2|-4k^{k-1}+|A_1|+(k+2)|A\setminus A_2|-4k^{k-1}\\
& \ge (k+2)|A|+|A_1|-8k^{k-1}\\
& >(k+2)|A|.\\
\end{align*}

\emph{Case 2.} $(u_2,k)=p$. We consider the following subcases.

\emph{Case 2a.} $m\in F$.  Since $E\ne\emptyset$, there exists $s\in E$. By Lemma \ref{imp2}, we have $|\Delta_{ss}|\ge |A_m|$. We denote $A'=A\setminus (A_m\cup A_s)$. We obtain
\begin{align*} 
|2\cdot A+k\cdot A| & \ge |2\cdot A_m+k\cdot A|+|2\cdot A_s+k\cdot A|+|2\cdot A'+k\cdot A'|\\
& \ge |2\cdot A_m+k\cdot A_1|+|2\cdot A_s+k\cdot A_s|+|\Delta_{ss}|+(k+2)|A'|-4k^{k-1}\\
& \ge 2|A_m|+k|A_1|-2k+(k+2)|A_s|-4k^{k-1}+|A_m|+(k+2)|A'|-4k^{k-1}\\
& \ge (k+2)|A|+|A_1|-8k^{k-1}-2k\\
& > (k+2)|A|-2k
\end{align*} 

\emph{Case 2b.} $m\in E$. Here we will consider separate cases when $|\hat{X}_m|\ge p$ and $|\hat{X}_m|< p$. Moreover, the case $|\hat{X}_m|\ge p$ we will divide in two subcases: $|A_1|\le q|A_m|$ and $|A_1|> q|A_m|$.

First we assume that $|\hat{X}_m|\ge p$ and $|A_1|\le q|A_m|$. 

Let $2\in F$. By Lemma \ref{imp2}, we have $|\Delta_{mm}|\ge |A_2|$. We denote $A'=A\setminus (A_2\cup A_m)$. We obtain
\begin{align*} 
|2\cdot A+k\cdot A| & \ge |2\cdot A_2+k\cdot A|+|2\cdot A_m+k\cdot A|+|2\cdot A'+k\cdot A'|\\
& \ge |2\cdot A_2+k\cdot A_1|+|2\cdot A_m+k\cdot A_m|+|\Delta_{mm}|+(k+2)|A'|-4k^{k-1}\\
& \ge 2|A_2|+k|A_1|-2k+(k+2)|A_m|-4k^{k-1}+|A_2|+(k+2)|A'|-4k^{k-1}\\
& \ge (k+2)|A|+|A_1|-8k^{k-1}-2k\\
& > (k+2)|A|-2k
\end{align*} 

If $2\in E$, by Lemma \ref{imp2}, we have $|\Delta_{22}|\ge |A_1|$. Thus
\begin{align*} 
|2\cdot A+k\cdot A| & \ge |2\cdot A_2+k\cdot A|+|2\cdot (A\setminus A_2)+k\cdot (A\setminus A_2)|\\
& = |2\cdot A_2+k\cdot A_2|+|\Delta_{22}|+(k+2)|A\setminus A_2|-4k^{k-1}\\
& \ge (k+2)|A_2|-4k^{k-1}+|A_1|+(k+2)|A\setminus A_2|-4k^{k-1}\\
& \ge (k+2)|A|+|A_1|-8k^{k-1}\\
& >(k+2)|A|.\\
\end{align*}  

Next we assume that $|\hat{X}_m|\ge p$ and $|A_1|> q|A_m|$. By Corollary \ref{cor2}, we have 
\[|2\cdot X_m+A| \ge |2\cdot X_m+A_1|\ge 2|X_m|+|\hat{X}_m||X_1|-2k.\] 
If $|\hat{X}_m|> p$, we obtain 
\begin{equation}\label{i}
|2\cdot X_m+A| \ge |2\cdot X_m+A_1|\ge 2|A_m|+(p+1)|A_1|-2k.
\end{equation}
If $|\hat{X}_m|=p$,  by Lemma \ref{l6}, we have $|(2\cdot \hat{X}_m+u_2)\setminus(2\cdot \hat{X}_m+u_1)|\ge 1$ and
\[
|(2\cdot X_m+A_2)\setminus(2\cdot X_m+A_1)|\ge |A_2||(2\cdot \hat{X}_m+u_2)\setminus(2\cdot \hat{X}_m+u_1)|\ge|A_2|.
\]
Thus
\begin{align}\label{ii}
|2\cdot X_m+A| &\ge|2\cdot X_m+A_1|+|(2\cdot X_m+A_2)\setminus(2\cdot X_m+A_1)|\\
&\ge|2\cdot X_m+A_1|+|A_2|\ge 2|A_m|+p|A_1|+|A_2|-2k.\notag
\end{align}

Now, let $2\in F$. We denote $A'=A\setminus (A_2\cup A_m)$. We have
\begin{align*} 
|2\cdot A+k\cdot A| & \ge |2\cdot A_2+k\cdot A|+|2\cdot A_m+k\cdot A|+|2\cdot A'+k\cdot A'|\\
& \ge |2\cdot A_2+k\cdot A_1|+|2\cdot X_m+A|+(k+2)|A'|-4k^{k-1}\\
& \ge 2|A_2|+k|A_1|-2k+(k+2)|A_m|+|A_2|-2k+(k+2)|A'|-4k^{k-1}\\
& \ge (k+2)|A|+|A_1|-8k^{k-1}-2k\\
& > (k+2)|A|-2k
\end{align*} 

If $2\in E$, by Lemma \ref{imp2}, we have $|\Delta_{22}|\ge q|A_m|$. Thus, if $A'=A\setminus (A_2\cup A_m)$, then
\begin{align*} 
|2\cdot A+k\cdot A| & \ge |2\cdot A_2+k\cdot A|+|2\cdot A_m+k\cdot A|+|2\cdot A'+k\cdot A'|\\
& \ge |2\cdot A_2+k\cdot A_2|+|\Delta_{22}|+|2\cdot X_m+A|+(k+2)|A'|-4k^{k-1}\\
& \ge(k+2)|A_2|-4k^{k-1}+q|A_m|+2|A_m|+p|A_1|+|A_2|-2k\\
&\quad +(k+2)|A'|-4k^{k-1}\\
& \ge(k+2)|A_2|+(k+2)|A'|+q|A_m|+2|A_m|+(p-1)q|A_m|+|A_1|\\
&\quad -8k^{k-1}-2k\\
& \ge (k+2)|A|+|A_1|-8k^{k-1}-2k\\
& > (k+2)|A|-2k
\end{align*} 

Finally, we assume that $|\hat{X}_m|<p$. By Lemma \ref{l6}, we have $|(2\cdot \hat{X}_m+u_1)\setminus(2\cdot \hat{X}_m+u_m)|\ge 1$ and
\[
|\Delta_{mm}|\ge|(2\cdot X_m+A_1)\setminus(2\cdot X_m+A_m)|\ge |A_1||(2\cdot \hat{X}_m+u_2)\setminus(2\cdot \hat{X}_m+u_1)|\ge|A_1|.
\]
Thus
\begin{align*} 
|2\cdot A+k\cdot A| & \ge |2\cdot A_m+k\cdot A|+|2\cdot (A\setminus A_m)+k\cdot (A\setminus A_m)|\\
& = |2\cdot A_m+k\cdot A_m|+|\Delta_{mm}|+(k+2)|A\setminus A_m|-4k^{k-1}\\
& \ge (k+2)|A_m|-4k^{k-1}+|A_1|+(k+2)|A\setminus A_m|-4k^{k-1}\\
& \ge (k+2)|A|+|A_1|-8k^{k-1}\\
& >(k+2)|A|.\\
\end{align*}

This ends the proof.

\bibliographystyle{amsplain}

\end{document}